\documentclass{amsart}
\usepackage{etex}
\usepackage[style=alphabetic,doi=false,isbn=false,url=false,maxbibnames=99,backend=bibtex]{biblatex}
\usepackage{fixltx2e}
\usepackage{tikz}
\usepackage{tikz-cd}
\usepackage{adjustbox}
\usepackage{enumitem}

\bibliography{stacksappendix}

\usepackage{amssymb,amsmath,amsfonts,amsthm,epsfig,amscd}
\usepackage{stmaryrd}
\usepackage[all,cmtip,poly]{xy}
\usepackage{color}

\newcommand{\codim}{\operatorname{codim}}
\newcommand{\trdeg}{\operatorname{trdeg}}

\def\iso{\buildrel \sim \over \longrightarrow}

\newtheorem{alemma}[subsection]{Lemma}
\newtheorem{alem}[subsection]{Lemma}

\newtheorem{acor}[subsection]{Corollary}

\newtheorem{ahyp}[subsection]{Hypothesis}

\theoremstyle{definition}

\newtheorem{adf}[subsection]{Definition}
\theoremstyle{remark}

\newtheorem{aremark}[subsection]{Remark}

\newtheorem{aexample}[subsection]{Example}

\def\numequation{\addtocounter{subsection}{1}\begin{equation}}
\def\nummultline{\addtocounter{subsubsection}{1}\begin{multline}}
\def\anumequation{\addtocounter{subsection}{1}\begin{equation}}
\def\anummultline{\addtocounter{subsection}{1}\begin{multline}}

\newif\iffinalrun
\iffinalrun
\else
\fi

\iffinalrun
  \newcommand{\need}[1]{}
  \newcommand{\mar}[1]{}
\else
  \newcommand{\need}[1]{{\tiny *** #1}}
  \newcommand{\mar}[1]{\marginpar{\raggedright\tiny 50footfrog #1}}
\fi

\newcommand{\A}{\AA}
\newcommand{\C}{\CC}

\newcommand{\Z}{\ZZ}

\newcommand{\m}{\frakm}

\renewcommand{\AA}{{\mathbb A}}

\newcommand{\CC}{{\mathbb C}}

\newcommand{\GG}{{\mathbb G}}

\newcommand{\ZZ}{{\mathbb Z}}

\renewcommand{\bf}{\ensuremath{\mathbf{f}}}

\newcommand{\cO}{{\mathcal O}}

\newcommand{\cT}{{\mathcal T}}
\newcommand{\cU}{{\mathcal U}}
\newcommand{\cV}{{\mathcal V}}
\newcommand{\cW}{{\mathcal W}}
\newcommand{\cX}{{\mathcal X}}
\newcommand{\cY}{{\mathcal Y}}
\newcommand{\cZ}{{\mathcal Z}}

\newcommand{\frakm}{\mathfrak{m}}

\DeclareMathOperator{\Spec}{Spec}

\DeclareMathOperator{\Spf}{Spf}

\newcommand{\Gm}{\GG_m}

\usepackage[bookmarksopen,bookmarksdepth=2]{hyperref}

\begin{document}
\title[Components of algebraic stacks]{Dimension theory and components of algebraic stacks}

\author[M. Emerton]{Matthew Emerton}\email{emerton@math.uchicago.edu}
\address{Department of Mathematics, University of Chicago,
5734 S.\ University Ave., Chicago, IL 60637, USA}

\author[T. Gee]{Toby Gee} \email{toby.gee@imperial.ac.uk} \address{Department of
  Mathematics, Imperial College London,
  London SW7 2AZ, UK}

\thanks{M.E.\ was supported in part by NSF grant DMS-1303450. T.G.\ was
  supported in part by a Leverhulme Prize, EPSRC grant EP/L025485/1,
  ERC Starting Grant 306326, and a Royal Society Wolfson Research
  Merit Award.}

\maketitle
\begin{abstract}
  We prove some basic results on the dimension theory of algebraic
  stacks, and on the multiplicities of their irreducible components,
  for which we do not know a reference.
\end{abstract}
\setcounter{tocdepth}{1}
\tableofcontents

In this short note we develop some basic results related to 
the notions of irreducible components and dimensions of locally
Noetherian algebraic stacks.  We work in the basic framework
of the Stacks Project \cite{stacks-project}; 
we also note that most of the results proved here have now been
incorporated
into~\cite[\href{http://stacks.math.columbia.edu/tag/0DQR}{Tag
  0DQR}]{stacks-project} (sometimes with weaker hypotheses than in
this note).

The main results on the dimension theory of algebraic stacks in the
literature that we are aware of are those of~\cite{MR3351957}, which
makes a study of the notions of codimension and relative dimension. We
make a more detailed examination of the notion of the dimension of an
algebraic stack at a point, and prove various results
relating the dimension of the fibres of a morphism at a point in the source
to the dimension of its source and target.  We also prove a result
(Lemma~\ref{lem: dimension formula} below) which
allow us (under suitable hypotheses) to compute the dimension of
an algebraic stack at a point in terms of a versal ring.

While we haven't always tried to optimise our results, we have 
largely tried to avoid making unnecessary hypotheses.  However, in some
of our results, in which we compare certain properties of an algebraic
stack to the properties of a versal ring to this
stack at a point, we have restricted our attention 
to the case of algebraic stacks that are locally finitely presented 
over a locally Noetherian scheme base, all of whose local rings are
$G$-rings. This gives us the convenience of having Artin approximation
available to compare the geometry of the versal ring to the geometry
of the stack itself.  However, this restrictive hypothesis
may not be necessary for the truth 
of all of the various statements that we prove. 
Since it is satisfied in the applications that we have in mind,
though,
we have been content to make it when it helps.

\smallskip

{\em Acknowledgements.}  We would like to thank our coauthors 
Ana Caraiani and David Savitt for their interest in this note,
as well as Brian Conrad and Johan de Jong for their valuable comments
on various parts of the manuscript.

\section{Multiplicities of components of algebraic stacks}
\label{sec:multiplicities}

If $X$ is a locally Noetherian scheme, then we may write $X$ (thought
of simply as a topological space) as a union 
of irreducible components, say $X = \cup T_i.$  Each irreducible
component is the closure of a unique generic point $\xi_i$, 
and the local ring $\mathcal O_{X,\xi_i}$ is a local Artin ring.
We may define the {\em multiplicity} $\mu_{T_i}(X)$ of $X$ along $T_i$ 
to be $\ell(\mathcal O_{X,\xi_i}).$

Our goal here is to generalise this definition to locally
Noetherian algebraic stacks.   If $\cX$ is such a stack, 
then it has an underlying topological space $|\cX|$
(see~\cite[\href{http://stacks.math.columbia.edu/tag/04Y8}{Definition 04Y8}]{stacks-project}),
which is locally
Noetherian
(by~\cite[\href{http://stacks.math.columbia.edu/tag/04Z8}{Definition 04Z8}]{stacks-project}),
and hence which may be written as a union of irreducible
components; we refer to these as the irreducible components of $\cX$.
If $\cX$ is quasi-separated, then $|\cX|$ is sober (by~\cite[Cor.\
5.7.2]{MR1771927}), 
but it need not be in the non-quasi-separated case.\footnote{We follow
      the Stacks Project in allowing our algebraic stacks
      to be non-quasi-separated.  However,
in the applications that we have in mind, the algebraic stacks involved
will in fact
be quasi-separated, and so the reader who prefers to restrict their
attention to the quasi-separated case will lose nothing by doing so.}
(Consider
for example the non-quasi-separated algebraic space $X := \A^1_{\C}/\Z$.)
Furthermore, there is no structure sheaf
on $|\cX|$ whose stalks can be used to define multiplicities.

In order to define the multiplicity of a component of $|\cX|$,
we use the fact that if $U \to \cX$ is a smooth surjection
from a scheme $U$ to $\cX$ (such a surjection exists,
since $\cX$ is an algebraic stack), it induces a surjection
$|U| \to |\cX|$ by~\cite[\href{http://stacks.math.columbia.edu/tag/04XI}{Tag 04XI}]{stacks-project} (here $|U|$ simply denotes the topological
space underlying $U$), and for each irreducible component
$T$ of $|\cX|$ there is an irreducible component $T'$ of
$|U|$ such that $T'$ maps into $T$ with dense image.
(See Lemma~\ref{lem:map of components} below for a proof.)

\begin{adf} 
\label{def:multiplicity}
We define $\mu_T(\cX) := \mu_{T'}(U).$
\end{adf}

Of course, we must check that this is independent
of the choice of chart $U$, and of the choice 
of irreducible component $T'$ mapping to $T$.
We begin by making this verification, as well as
proving Lemma~\ref{lem:map of components}.

\begin{alemma}
\label{lem:map of components}
If $U \to \cX$ is a smooth morphism from a scheme
onto a locally Noetherian algebraic stack $\cX$, then the closure of the image of any irreducible
component of $|U|$ is an irreducible component of $|\cX|$.
If this morphism is furthermore surjective,
then all irreducible components of $|\cX|$ are obtained in this way.
\end{alemma}
\begin{proof}
This is easily verified, using the fact that $|U| \to |\cX|$ is
continuous and open by~\cite[\href{http://stacks.math.columbia.edu/tag/04XL}{Lem.~04XL}]{stacks-project}, and furthermore surjective if $U \to \cX$ is,
once one recalls that the irreducible components of a locally
Noetherian topological space can be
characterised as being the closures of irreducible open subsets of the space.
\end{proof}

The preceding lemma applies in particular in the case of smooth morphisms
between locally Noetherian schemes.    This particular case is
implicitly invoked in the statement of the following lemma.

\begin{alemma}
\label{lem:multiplicities}
If $U \to X$ is a smooth morphism of locally Noetherian schemes,
and if $T'$ is an irreducible component of $U$, with $T$ denoting
the irreducible component of $X$ obtained as the closure of the
image of $T'$, then $\mu_{T'}(U) = \mu_{T}(X).$
\end{alemma}
\begin{proof}Write $\xi'$ for the generic point of $T'$, and $\xi$ for the
  generic point of $T$, so that we need to show that $\ell(\mathcal
  O_{X,\xi})=\ell(\mathcal O_{U,\xi'})$.

  Let $n=\ell(\cO_{X,\xi})$, and choose a sequence $\cO_{X,\xi}=I_0\supset
  I_1\supset\dots\supset I_n=0$ with $I_i/I_{i+1}\cong
  \cO_{X,\xi}/\m_{X,\xi}$. The map $\cO_{X,\xi}\to\cO_{U,\xi'}$ is
  flat, 
so that we have \[I_i\cO_{U,\xi'}/I_{i+1}\cO_{U,\xi'}\cong
  (I_i/I_{i+1})\otimes_{\cO_{X,\xi}}\cO_{U,\xi'} \cong
  \cO_{U,\xi'}/\m_{X,\xi}\cO_{U,\xi'},\]so it suffices to show that
  $\m_{X,\xi}\cO_{U,\xi'}=\m_{U,\xi'}$, or in other words that
  $\cO_{U,\xi'}/\m_{X,\xi}\cO_{U,\xi'}$ is reduced.

Since the map $U\to X$ is smooth,
so is its base-change $U_{\xi} \to \Spec \kappa(\xi).$   As $U_{\xi}$ is a
smooth scheme over a field, it is reduced, and thus so its local ring
at any point.  In particular,
  $\cO_{U,\xi'}/\m_{X,\xi}\cO_{U,\xi'}$,
which is naturally identified with the local ring of $U_{\xi}$ at $\xi'$,
is reduced, as required.
%
\end{proof}

Using this result, we may show that notion of multiplicity
given in
Definition~\ref{def:multiplicity} is in fact well-defined.

\begin{alemma}
If $U_1 \to \cX$ and $U_2 \to \cX$ are two smooth surjections from
schemes to the locally Noetherian algebraic stack $\cX$,
and $T_1'$ and $T_2'$ are irreducible components of $|U_1|$
and $|U_2|$ respectively, the closures of whose images
are both equal to the same irreducible component $T$ of $|\cX|$,
then $\mu_{T_1'}(U_1) = \mu_{T_2'}(U_2)$.
\end{alemma}
\begin{proof}
Let $V_1$ and $V_2$ be dense subsets of $T_1'$ and $T'_2$, respectively,
that are open in $U_1$ and $U_2$ respectively.
The images of $|V_1|$ and $|V_2|$ in $|\cX|$ are non-empty open 
subsets of the irreducible subset $T$, and therefore have non-empty
intersection.
By~\cite[\href{http://stacks.math.columbia.edu/tag/04XH}{Tag
    04XH}]{stacks-project}, the map $|V_1\times_\cX V_2|\to
  |V_1|\times_{|\cX|}|V_2|$ is surjective,
and consequently $V_1\times_{\cX} V_2$ is a non-empty algebraic
space; we may therefore choose an \'etale surjection 
$V \to V_1\times_{\cX} V_2$ whose source is a (non-empty) scheme.
If we let $T'$ be any irreducible component of $V$,
then Lemma~\ref{lem:map of components} shows that the closure of
the image of $T'$ in $U_1$ (respectively $U_2$) is equal to $T'_1$
(respectively $T'_2$).

Applying
  Lemma~\ref{lem:multiplicities} twice we find
  that \[\mu_{T_1'}(U_1)=\mu_{T'}(V)=\mu_{T_2'}(U_2),\]as required.
\end{proof} 

It will be convenient to have a comparison between the notion of multiplicity
of an irreducible component given by Definition~\ref{def:multiplicity}
and the related notion of multiplicities of irreducible
components of (the spectra of) versal rings of $\cX$ at finite type points.
In order to have a robust theory of versal rings at finite type points,
we assume for the remainder of this note that $\cX$ is locally
of finite presentation
over a locally Noetherian scheme $S$, all of whose local rings
are $G$-rings.  (This hypothesis on the local rings
may not be necessary for all the assertions that follow, but it makes the
arguments straightforward, and in any case seems to be necessary
for the actual comparison of multiplicities. We also note
that condition this is equivalent to the apparently weaker condition
that the local rings of $S$ at finite type points are $G$-rings;
indeed, the finite type points are dense in 
$S$~\cite[\href{http://stacks.math.columbia.edu/tag/02J4}{Lem.~02J4}]{stacks-project},
and essentially by definition, any localization of a $G$-ring is again a
$G$-ring.)

We begin by recalling the following standard consequence of Artin approximation.
\begin{alemma}
  \label{lem: Artin approximation by smooth morphism}
  Let $\cX$ be an algebraic stack locally of finite presentation
over a locally Noetherian scheme $S$,
all of whose local rings 
are $G$-rings,
and let $x: \Spec k \to \cX$ be a morphism whose source is the spectrum
of a field of finite type over $\cO_S$.

If $A_x$ is a versal ring to $\cX$ at $x$, then we may find a smooth
morphism $U\to\cX$ whose source is a scheme, containing a point $u \in U$ of residue field $k$,
such that the induced morphism $u = \Spec k  \to U \to \cX$
coincides with the given morphism $x$,
and such that there is an isomorphism $\widehat{\cO}_{U,u} \cong
A_x$ compatible with the versal morphism $\Spf A_x \to \cX$
and the induced morphism $\Spf \widehat{\cO}_{U,u} \to U \to \cX$.
\end{alemma}
\begin{proof}Since $\cX$ is an algebraic stack,
the versal morphism $\Spf A_x \to \cX$ is effective,
i.e.\ can be promoted to a morphism $\Spec A_x \to \cX$
\cite[\href{http://stacks.math.columbia.edu/tag/07X8}{Lem.~07X8}]{stacks-project}.
By assumption $\cX$ is locally of finite presentation over $S$,
and hence limit preserving \cite[Lem.~2.1.9]{EGstacktheoreticimages},
and so Artin approximation
(see \cite[\href{http://stacks.math.columbia.edu/tag/07XH}{Lem.~07XH}]{stacks-project} and its proof)
shows that we may find a morphism $U \to \cX$ with source a finite type 
$S$-scheme, containing a point $u \in U$ of residue field $k$,
satisfying all of the required properties except possibly the
smoothness of $U\to\cX$.

Since $\cX$ is an algebraic stack, we see that if we replace
$U$ by a sufficiently small neighbourhood of $u$, we may in addition
assume that $U \to \cX$ is smooth
(see e.g.\ \cite[Lem.~2.4.7~(4)]{EGstacktheoreticimages}), as required. 
\end{proof}

\begin{alemma}
\label{lem:branches}
Let $\cX$ be an algebraic stack locally of finite presentation
over a locally Noetherian scheme $S$,
all of whose local rings 
are $G$-rings,
and let $x: \Spec k \to \cX$ be a morphism whose source is the spectrum
of a field of finite type over $\cO_S$.
If $A_x$ and $A'_x$ are two versal rings to $\cX$ at $x$,
then the multi-sets of irreducible components of $\Spec A_x$
and of $\Spec A'_x$
{\em (}in which each component is counted with its multiplicity{\em )},
are in canonical bijection.

Furthermore, there is a natural surjection from the 
set of irreducible components of each of $\Spec A_x$ and $\Spec A'_x$
to the set of irreducible components of $|\cX|$ containing the
class of $x$ in $|\cX|$; this surjection sends components
that correspond by the above bijection to the same component
of $|\cX|$; and this surjection preserves multiplicities.
\end{alemma}
\begin{proof}
By Lemma~\ref{lem: Artin approximation by smooth morphism} we can find
smooth morphisms $U,U' \to \cX$ whose sources are schemes, and  points
$u,u'$ of $U,U'$ respectively, both with
residue field $k$, such that the induced morphisms $\widehat{\cO}_{U,u}
\to U \to \cX$ and $\widehat{\cO}_{U',u'}
\to U' \to \cX$ can be identified respectively with the versal
morphisms $\Spf A_x\to\cX$ and $\Spf A'_x \to \cX$.
We then form the fibre product $U'' := U \times_{\cX} U'$; this is an algebraic
space over $S$, and the two monomorphisms $u = \Spec k \to U$
and $u' = \Spec k \to U'$ induce a monomorphism $u'' = \Spec k \to U''.$
Following \cite[Prop.~2.2.14, Def.~2.2.16]{EGstacktheoreticimages},
we consider the complete local ring $\widehat{\cO}_{U'',u}$
of $U''$ at $u''$.

Since $U'' \to U$ is smooth, we see that the induced morphism
$A_x = \widehat{\cO}_{U,u} \to \widehat{\cO}_{U'',u''}$ induces
a  smooth morphism of representable functors,
in the sense
of~\cite[\href{http://stacks.math.columbia.edu/tag/06HG}
{Def.~06HG}]{stacks-project},
and hence, by~\cite[\href{http://stacks.math.columbia.edu/tag/06HL}
{Lem.~06HL}]{stacks-project},
we see that $\widehat{\cO}_{U'',u''}$ is a formal power series
ring over $A_x$.  Similarly, it is a formal power series
ring over $A'_x$.  Recall that if
$A$ is a complete local ring and $B$ is a formal
power series ring in finitely many variables over $A$,
then the irreducible components of $\Spec B$ are in a natural multiplicity
preserving bijection with the irreducible components of $\Spec A$.
Thus, we obtain multiplicity preserving bijections
between the multi-sets of irreducible components of each
of $\Spec A_x$ and $\Spec A'_x$ with the multi-set
of irreducible components of $\Spec \widehat{\cO}_{U'',u''}$,
and hence between these two multi-sets themselves.

The morphism $\Spec A_x \to \cX$ factors through $U$,
and the scheme-theoretic image of each irreducible component 
of $\Spec A_x$ is an irreducible component of $U$ (as follows
from the facts that $\Spec A_x \to U$ is flat,
and that flat morphisms satisfy the {\em going-down theorem}).
Composing with the natural map
from the set of irreducible components of $U$ to the set
of irreducible components of $\cX$, 
we obtain a morphism from the set of irreducible components of $\Spec A_x$
to the set of irreducible components of $|\cX|$.
A consideration of the commutative diagram
$$\xymatrix{|\Spec \cO_{U'',u''}| \ar[r]\ar[d] & |U''| \ar[d] \\
|\Spec A_x| \ar[r] & |\cX| }$$
and of the analogous diagram with $A'_x$ in place of $A_x$,
shows that this map, and the corresponding map for $A'_x$,
are compatible with the bijection constructed
above between the irreducible components of $\Spec A_x$ and 
the irreducible components of $\Spec A'_x$.

It remains to show that this map, from the irreducible components
of $\Spec A_x$ to those of $\cX$, is multiplicity preserving.
A consideration of the definition of the multiplicity of an irreducible
component of $\cX$, and of the preceding constructions,
shows that it suffices to show that the map from the set of 
irreducible components of $\Spec \widehat{\cO}_{U'',u''}$ to
the set of irreducible components of $U''$, given by taking Zariski closures,
is multiplicity preserving.  As we will see, this follows from the assumption
that the local rings of $S$ 
are $G$-rings. 

More precisely, noting that it suffices to compare these multiplicities
after making an \'etale base-change, we may replace $U''$ 
by a scheme which covers it via an \'etale map, and hence assume
that $U''$ itself is a scheme, so that the local ring
$\cO_{U'',u''}$ is defined. (Alternatively, we could apply
Artin approximation to the versal morphism
$\Spf \widehat{\cO}_{U'',u''} \to U''$, so as to replace $U''$ by a scheme.)  
The scheme $U''$ is of finite type over $S$,
and hence the local ring $\cO_{U'',u''}$ is a $G$-ring.
Let $\mathfrak p$ be a minimal prime ideal of $\cO_{U'',u''}$,
corresponding to an irreducible component of $U''$ passing
through $u''$, and let $\mathfrak q$ be a minimal prime
of $\widehat{\cO}_{U'',u''}$ lying over $\mathfrak p$
(corresponding to an irreducible component of $\Spec \widehat{\cO}_{U'',u''}$
whose closure in $U''$ is the irreducible component corresponding to
$\mathfrak p$);
we have to show that the length
of $(\widehat{\cO}_{U'',u''})_{\mathfrak q}$ is equal to the length of
$(\cO_{U'',u''})_{\mathfrak p}$.
Since $\mathfrak q$ lies over $\mathfrak p,$
there is a natural isomorphism
$$ (\widehat{\cO}_{U'',u''})_{\mathfrak q}
\cong 
\bigl(\widehat{\cO}_{U'',u''} \otimes
(\cO_{U'',u''})_{\mathfrak p}\bigr)_{\mathfrak q}.$$
Now if $\ell$ is the length of 
$(\cO_{U'',u''})_{\mathfrak p},$
then we may find a filtration of length $\ell$ on
$(\cO_{U'',u''})_{\mathfrak p}$, each of whose graded
pieces is isomorphic to $\kappa(\mathfrak p)$.
This induces a corresponding filtration on
$\bigl(\widehat{\cO}_{U'',u''} \otimes
(\cO_{U'',u''})_{\mathfrak p}\bigr)_{\mathfrak q},$
each of whose graded pieces is isomorphic to 
$\bigl(\widehat{\cO}_{U'',u''} \otimes
\kappa(\mathfrak p)\bigr)_{\mathfrak q}.$
Since $\cO_{U'',u''}$ is a $G$-ring,
the formal fibre 
$\widehat{\cO}_{U'',u''} \otimes
\kappa(\mathfrak p)$ is regular.
Since $\mathfrak q$ is a minimal prime in this ring,
the localization
$\bigl(\widehat{\cO}_{U'',u''} \otimes
\kappa(\mathfrak p)\bigr)_{\mathfrak q}$
is thus a field, and hence equal to $\kappa(\mathfrak q)$.
We conclude that $(\widehat{\cO}_{U'',u''})_{\mathfrak q}$
has length $\ell$, as required.

\end{proof}

\begin{adf}
If $\cX$ is an algebraic stack locally of finite presentation
over a locally Noetherian scheme $S$ all of whose local rings are $G$-rings,
if $x: \Spec k \to \cX$ is a morphism whose source is the spectrum
of a field of finite type over $\cO_S$,
and if $A_x$ is a versal ring to $\cX$ at $x$,
then we define the set of \emph{formal branches of $\cX$ through $x$} to be 
the set of irreducible components of $\Spec A_x$,
and we define the multiplicity of a branch to be the multiplicity
of the corresponding component in $\Spec A_x$.
\end{adf}

Lemma~\ref{lem:branches} shows, in the context of the preceding
definition, that the set of formal branches of $\cX$ through $x$,
and their multiplicities, are well-defined independently of
the choice of versal ring used to compute them.  It also
shows that there is a natural map from the set 
of formal branches of $\cX$ through $x$ to the set of irreducible
components of $|\cX|$ containing the class of $x$,
and that this map preserves multiplicities.


%

\medskip

As a closing remark,
we note that
it is sometimes convenient to think of an irreducible component of $\cX$ as a
closed substack. To this end, if $\cT$ is an irreducible component of $\cX$, i.e.\
an irreducible component of $|\cX|$, then we endow $\cT$ with its induced reduced
substack structure
(see~\cite[\href{http://stacks.math.columbia.edu/tag/050C}{Def.~050C}]{stacks-project}).

 \section{Dimension theory of algebraic stacks}\label{sec:appendix on
   dimension of algebraic stacks}
%
In this section we discuss some concepts related to the dimension
theory of locally Noetherian algebraic stacks.  
Since we intend to make arguments with them, it will be helpful
to recall the basic definitions related to dimensions,
beginning with the case of schemes, and then the case of algebraic spaces.

\begin{adf}
If $X$ is a scheme,
then we define the dimension $\dim(X)$ of $X$
to be the Krull dimension of the 
topological space underlying $X$,
while if $x$ is a point of $X$,
then we define the dimension $\dim_x (X)$ of $X$ at $x$ to be the 
minimum of the dimensions of the open subsets $U$ of $X$ containing
$x$~\cite[\href{http://stacks.math.columbia.edu/tag/04MT}{Def.~04MT}]
{stacks-project}.
One has the relation $\dim(X) = \sup_{x \in X}
\dim_x(X)$~\cite[\href{http://stacks.math.columbia.edu/tag/04MU}{Lem.~04MU}]
{stacks-project}.

If $X$ is locally Noetherian, then $\dim_x(X)$ coincides with the supremum 
of the dimensions at $x$
of the irreducible components of $X$ passing through $x$.
\end{adf}

\begin{adf}
\label{def:dimension for algebraic spaces}
If $X$ is an algebraic space and $x \in |X|$,
then we define $\dim_x X = \dim_u U,$ where $U$ is any scheme
admitting an \'etale surjection $U \to X$,
and $u\in U$ is any point lying over
$x$~\cite[\href{http://stacks.math.columbia.edu/tag/04N5}{Def.~04N5}]
{stacks-project}. We set $\dim(X) = \sup_{x \in |X|}
\dim_x(X)$.
\end{adf}

\begin{aremark}
In general, the dimension of the algebraic space $X$ at a point $x$
may not coincide with the dimension of the underlying topological space
$|X|$ at $x$.  E.g.\ if $k$ is a field of characteristic zero and
$X =  \A^1_k / \Z$, then $X$ has dimension $1$ (the dimension
of $\A^1_k$) at each of its points,
while $|X|$ has the indiscrete topology, and hence is of Krull
dimension zero.   On the other hand,
in~\cite[\href{http://stacks.math.columbia.edu/tag/02Z8}{Ex.~02Z8}]
{stacks-project} 
there is given an example of an algebraic space
which is of dimension $0$ at each of its points, while $|X|$ is
irreducible of Krull dimension $1$, and admits a generic point (so that the
dimension of $|X|$ at any of its points is $1$); see also the discussion
of this example 
in~\cite[\href{http://stacks.math.columbia.edu/tag/04N3}{Tag~04N3}]
{stacks-project} 

On the other hand, if $X$ is a {\em decent} algebraic space, in the sense
of~\cite[\href{http://stacks.math.columbia.edu/tag/03I8}{Def.~03I8}]
{stacks-project}
(in particular, if $X$ is quasi-separated;
see~\cite[\href{http://stacks.math.columbia.edu/tag/03I7}{Def.~03I7}]
{stacks-project}),
then in fact the dimension of $X$ at $x$ does coincide with the dimension
of $|X|$ at $x$;
see~\cite[\href{http://stacks.math.columbia.edu/tag/0A4J}{Lem.~0A4J}]
{stacks-project}.
\end{aremark}

In order to  define the dimension of an algebraic stack,
it will be useful to first have the notion of the relative dimension,
at a point in the source, 
of a morphism whose source is an algebraic space,
and whose target is an algebraic stack.  The definition is slightly
involved, just because (unlike in the case of schemes) the points of an algebraic stack, or an algebraic
space, are not describable as morphisms from the spectrum of a field, but only as equivalence classes of such.

\begin{adf}
\label{def:relative dimension}
If $f:T \to \cX$ is a locally of finite type morphism from an algebraic space
to an algebraic stack,
and if $t \in |T|$ is a point with image $x \in | \cX|$, then we define 
{\em the relative dimension} of $f$ at $t$, denoted
$\dim_t(T_x),$
as follows: 
choose a morphism $\Spec k \to \cX$, with source the spectrum of 
a field, which represents $x$, and choose a point $t' \in | T\times_{\cX}
\Spec k |$ mapping to $t$ under the projection to $|T|$
(such a point $t'$ exists,
by~\cite[\href{http://stacks.math.columbia.edu/tag/04XH}{Lem.~04XH}]
{stacks-project}); then
$$\dim_t(T_x) := \dim_{t'}(T \times_{\cX} \Spec k ).$$
(Note that since $T$ is an algebraic space and $\cX$ is an algebraic stack,
the fibre product $T\times_{\cX} \Spec k$ is an algebraic space,
and so the quantity on the right hand side of this proposed definition
is in fact defined, by Definition~\ref{def:dimension for algebraic spaces}.)
\end{adf}

\begin{aremark}
\label{rem:relative dimension}
(1) 
One easily verifies (for example, by using the invariance
of the relative dimension of locally of finite type morphisms of schemes
under base-change;
see e.g.~\cite[\href{http://stacks.math.columbia.edu/tag/02FY}{Lem.~02FY}]
{stacks-project})
that $\dim_t(T_x)$ is well-defined, independently of the choices
used to compute it. 

(2)
In the case that $\cX$ is also an algebraic space,
it is straightforward to confirm that this definition agrees with
the definition of relative dimension given 
in~\cite[\href{http://stacks.math.columbia.edu/tag/04NM}{Def.~04NM~(3)}]
{stacks-project}.
\end{aremark}

We next recall the following lemma, on which the definition of
the dimension of a locally Noetherian algebraic stack is founded.

\begin{alemma}
\label{lem:behaviour of dimensions w.r.t. smooth morphisms}
If $f: U \to X$ is a smooth morphism of locally Noetherian algebraic
spaces, and
if $u \in |U|$ with image $x \in |X|$,
then $$\dim_u (U) = \dim_x(X) + \dim_{u} (U_x)$$
{\em (}where of course $\dim_u (U_x)$ is defined via
Definition~{\em \ref{def:relative dimension})}.
\end{alemma}
\begin{proof}
See~\cite[\href{http://stacks.math.columbia.edu/tag/0AFI}{Lem.~0AFI}]
{stacks-project},
noting that the definition of $\dim_u (U_x)$ used here coincides with
the definition used there, by Remark~\ref{rem:relative dimension}~(2).
\end{proof}

\begin{adf}
\label{def:dimension for stacks}
If $\cX$ is a locally Noetherian algebraic stack,
and $x \in |\cX|$, then we  define the dimension $\dim_x(\cX)$
of $\cX$ at $x$ as follows:
let $U \to \cX$ be a smooth morphism 
from a scheme (or, more generally, from an algebraic space) to $\cX$
containing $x$ in its image,
let $u$ be any point of $|U|$ mapping to $x$,
and define
$$\dim_x(\cX) :=  \dim_u (U) - \dim_{u}( U_x)$$
(where the relative dimension $\dim_u(U_x)$ is defined
by Definition~\ref{def:relative dimension}).
\end{adf}

\begin{aremark}
	\label{rem:computing dims}
The preceding definition is justified by the formula of
Lemma~\ref{lem:behaviour of dimensions w.r.t. smooth morphisms},
and one can use that lemma to verify that $\dim_x(\cX)$ is well-defined,
independently of the choices used to compute it.
Alternatively (employing the notation of the definition, and choosing 
$U$ to be a scheme),
one can compute $\dim_u(U_x)$ by choosing
the representative of $x$ to  be the composite
$\Spec \kappa(u) \to U \to \cX$, where the first morphism is the canonical
one with image $u \in U$. 
Then, if we write $R := U\times_{\cX} U$, and let $e: U \to R$ denote the
diagonal morphism, the invariance of relative dimension under base-change
shows that $\dim_u \bigl(U_x) = \dim_{e(u)}(R_u),$
and thus the preceding definition of $\dim_x (\cX)$ coincides
with the definition as $\dim_u (U) - \dim_{e(u)}(R_u)$ given
in~\cite[\href{http://stacks.math.columbia.edu/tag/0AFN}{Def.~0AFN}]
{stacks-project}, 
which is shown to be independent of choices 
in~\cite[\href{http://stacks.math.columbia.edu/tag/0AFM}{Lem.~0AFM}]
{stacks-project}. 
\end{aremark}

\begin{aremark}
For Deligne--Mumford stacks which are suitably decent
(e.g.\ quasi-separated),
it will again be the case that $\dim_x(\cX)$ coincides with the topologically
defined quantity $\dim_x |\cX|$.  However, for more general Artin stacks, 
this will typically not be the case.  For example, if $\cX := [\A^1/\Gm]$
(over some field, with the quotient being taken with
respect to the usual multiplication action of $\Gm$ on $\A^1$),
then  $|\cX|$ has two points, one the specialisation of the other (corresponding
to the two orbits of $\Gm$ on $\A^1$), and hence is of dimension $1$ as
a topological space; but $\dim_x (\cX) = 0$ for both points $x \in |\cX|$.
(An even more extreme example is given by the classifying space 
$[\Spec k/\Gm]$, whose dimension at its unique point
is equal to~$-1$.)
\end{aremark}

We can now extend Definition~\ref{def:relative dimension} to the context
of (locally of finite type) 
morphisms between (locally Noetherian) algebraic stacks.

\begin{adf}
\label{def:relative dimension for stacks}
If $f:\cT \to \cX$ is a locally of finite type morphism between 
locally Noetherian algebraic stacks,
and if $t \in |\cT|$ is a point with image $x \in | \cX|$, then we define 
the {\em relative dimension} of $f$ at $t$, denoted
$\dim_t(\cX_x),$
as follows: 
choose a morphism $\Spec k \to \cX$, with source the spectrum of 
a field, which represents $x$, and choose a point $t' \in | \cT\times_{\cX}
\Spec k |$ mapping to $t$ under the projection to $|\cT|$
(such a point $t'$ exists,
by~\cite[\href{http://stacks.math.columbia.edu/tag/04XH}{Lem.~04XH}]
{stacks-project}); then
$$\dim_t(\cT_x) := \dim_{t'}(\cT \times_{\cX} \Spec k ).$$
(Note that since $\cT$ is an algebraic stack and $\cX$ is an algebraic
stack, 
the fibre product $\cT\times_{\cX} \Spec k$ is an algebraic stack,
which is locally Noetherian 
by~\cite[\href{http://stacks.math.columbia.edu/tag/06R6}{Lem.~06R6}]{stacks-project}.
Thus the quantity on the right side of this proposed definition
is defined by Definition~\ref{def:dimension for stacks}.)
\end{adf}

\begin{aremark}
Standard manipulations show that $\dim_t(\cT_x)$ is well-defined,
independently of the choices made to compute it.
\end{aremark}

\bigskip

We now establish some basic properties of relative dimension, which
are obvious generalisations of the corresponding statements in the 
case of morphisms of schemes.

\begin{alemma}
\label{lem:base-change invariance of relative dimension}
Suppose given 
a Cartesian square of morphisms of locally Noetherian stacks
$$\xymatrix{\cT' \ar[d]\ar[r] & \cT \ar[d] \\
\cX' \ar[r] & \cX}$$
in which the vertical morphisms are locally of finite type.
If $t' \in |\cT'|$, 
with images $t$, $x'$, and $x$ in $|\cT|$, $|\cX'|$, and $|\cX|$ 
respectively, then $\dim_{t'}(\cT'_{x'}) = \dim_{t}(\cT_x).$
\end{alemma}
\begin{proof}Both sides can (by definition) be computed as the
  dimension of the same fibre product.
\end{proof}

\begin{alemma}
\label{lem:behaviour of dimensions w.r.t. smooth morphisms; stacky case}
If $f: \cU \to \cX$ is a smooth morphism of locally Noetherian algebraic
stacks, and
if $u \in |\cU|$ with image $x \in |\cX|$,
then $$\dim_u (\cU) = \dim_x(\cX) + \dim_{u} (\cU_x).$$
\end{alemma}
\begin{proof}
Choose a smooth surjective morphism $V \to \cU$ whose source
is a scheme, and let $v\in |V|$ be a point mapping to $u$.
Then the composite $V \to \cU \to \cX$ is also smooth,
and by definition we have $\dim_x(\cX) = \dim_v(V) - \dim_v(V_x),$
while $\dim_u(\cU) = \dim_v(V) - \dim_v(V_u).$
Thus
$$\dim_u(\cU) - \dim_x(\cX) = \dim_v (V_x) - \dim_v (V_u).$$

Choose a representative $\Spec k \to \cX$ of $x$
and choose a point $v' \in | V \times_{\cX} \Spec k|$ lying over
$v$, with image $u'$ in $|\cU\times_{\cX} \Spec k|$;
then by definition
$\dim_u(\cU_x) = \dim_{u'}(\cU\times_{\cX} \Spec k),$
and
$\dim_v(V_x) = \dim_{v'}(V\times_{\cX} \Spec k).$

Now $V\times_{\cX} \Spec k \to \cU\times_{\cX}\Spec k$
is a smooth surjective morphism (being the base-change 
of such a morphism) whose source is an algebraic space
(since $V$ and $\Spec k$ are schemes, and $\cX$
is an algebraic stack).  Thus, again by definition,
we have
\begin{multline*}
\dim_{u'}(\cU\times_{\cX} \Spec k) =
\dim_{v'}(V\times_{\cX} \Spec k) -
\dim_{v'}\bigl( (V\times_{\cX} \Spec k)_{u'})
\\
= \dim_v(V_x) - 
\dim_{v'}\bigl( (V\times_{\cX} \Spec k)_{u'}).
\end{multline*}
Now $V\times_{\cX} \Spec k \iso V\times_{\cU} (\cU\times_{\cX} \Spec k),$
and so
Lemma~\ref{lem:base-change invariance of relative dimension}
shows that
$\dim_{v'}\bigl( (V\times_{\cX} \Spec k)_{u'})  = \dim_v(V_u).$
Putting everything together, we find that
$$
\dim_u(\cU) - \dim_x(\cX) = 
\dim_u(\cU_x),$$
as required.
\end{proof}

\begin{alemma}
\label{lem:relative dimension is semi-continuous}
Let $f: \cT \to \cX$ be a locally of finite type morphism of algebraic stacks.
\begin{enumerate}
\item
The function $t \mapsto \dim_t(\cT_{f(t)})$ is upper semi-continuous 
on~$|\cT|$.
\item If $f$ is smooth, then
the function $t \mapsto \dim_t(\cT_{f(t)})$ is locally constant
on~$|\cT|$.
\end{enumerate}
\end{alemma}
\begin{proof}
	Suppose to begin with that $\cT$ is a scheme $T$,
	let $U \to \cX$ be a smooth surjective morphism whose source
	is a scheme, and let $T':=T\times_{\cX}U$. Let $f': T' \to U$ be
	the pull-back of $f$ over $U$,
	and let $g: T' \to T$ be the projection.

      	Lemma~\ref{lem:base-change
		invariance of relative dimension}
	shows that $\dim_{t'}(T'_{f'(t')}) = \dim_{g(t')}(T_{f(g(t'))}),$
	for $t' \in T'$, while, 
	since $g$ is smooth and surjective (being the base-change
	of a smooth surjective morphism) the map induced by~$g$ on underlying
	topological spaces is continuous and open
(by~\cite[\href{http://stacks.math.columbia.edu/tag/04XL}{Lem.~04XL}]{stacks-project}), and surjective.
        Thus it suffices to note that part~(1) for the morphism $f'$
	follows from~\cite[\href{http://stacks.math.columbia.edu/tag/04NT}{Tag
  04NT}]{stacks-project}, and part~(2)
from either
of~\cite[\href{http://stacks.math.columbia.edu/tag/02NM}{Lem.~02NM}]{stacks-project}
or~\cite[\href{http://stacks.math.columbia.edu/tag/02G1}{Lem.~02G1}]{stacks-project}
(each of which gives the result for schemes, from which
the analogous results for algebraic spaces can
be deduced exactly as in~\cite[\href{http://stacks.math.columbia.edu/tag/04NT}{Tag
  04NT}]{stacks-project}).

	Now return to the general case,
	and choose a smooth surjective morphism
	$h:V \to \cT$ whose source is a scheme.  
	If $v \in V$, then, essentially by definition,
	we have 
	$$\dim_{h(v)}(\cT_{f(h(v))}) = 
	\dim_{v}(V_{f(h(v))}) - \dim_{v}(V_{h(v)}).$$
	Since $V$ is a scheme, we have proved that the first 
	of the terms on the right hand side of this equality
	is upper semi-continuous (and even locally
	constant if $f$ is smooth), while the second term is 
	in fact locally constant. 
	Thus their difference is upper semi-continuous
	(and locally constant if $f$ is smooth),
	and hence the function
	$\dim_{h(v)}(\cT_{f(h(v))})$
        is upper semi-continuous on $|V|$ (and locally
	constant if $f$ is smooth).  	
        Since the morphism $|V| \to |\cT|$ is open and surjective,
	the lemma follows.
\end{proof}

Before continuing with our development,
we prove two lemmas related to the dimension theory of schemes.

To put the first lemma in context,
we note that if $X$ is a finite-dimensional scheme, then since $\dim X$
is 
equal to the supremum of the dimensions $\dim_x X$,
there exists a point $x \in X$ such that $\dim_x X = \dim X$. 
The following lemma shows that we may furthermore take the point
$x$ to be of finite type.

\begin{alemma}
	\label{lem:dimension achieved by finite type point}
	If $X$ is a finite-dimensional scheme,
	then there exists a closed {\em (}and hence
       	finite type{\em )}  point $x \in X$
	such that $\dim_x X = \dim X$.
\end{alemma}
\begin{proof}
	Let $d = \dim X$,
	and choose a maximal strictly decreasing
	chain of irreducible closed subsets of $X$,
	say
	\anumequation
	\label{eqn:maximal chain}
	Z_0 \supset Z_1 \supset \cdots \supset Z_d.
\end{equation}
        The subset $Z_d$ is a minimal irreducible closed subset of $X$,
	and thus any point of $Z_d$ is a generic point of $Z_d$.
	Since the underlying topological space of the scheme $X$ is sober,
	we conclude that $Z_d$ is a singleton, consisting of a single 
	closed point $x \in X$.
%
If $U$ is 
	any neighbourhood of $x$, then
	the chain 
	$$
	U\cap Z_0 \supset U\cap Z_1 \supset \cdots \supset U\cap Z_d = Z_d =
	\{x\}
	$$
	is then a strictly descending chain of irreducible
	closed subsets of $U$, showing that $\dim U \geq d$.
	Thus we find that $\dim_x X \geq d$.  The other inequality
	being obvious, the lemma is proved.
\end{proof}

The next lemma shows that $\dim_x X$ is a {\em constant} function
on an irreducible scheme satisfying some mild additional hypotheses.
(See Lemma~\ref{lem:equicodimensionality of opens} below
for a related result.)

\begin{alemma}
	\label{lem:constancy of dimension}
	If $X$ is an irreducible, Jacobson, catenary, and locally Noetherian
	scheme of finite dimension,
	then $\dim U = \dim X$ for every
	non-empty open subset $U$ of $X$.
	Equivalently, $\dim_x X$ is a constant function on $X$.
\end{alemma}
\begin{proof}
	The equivalence of the two claims follows directly from the
	definitions.   Suppose, then, that $U\subset X$ is a non-empty open
	subset.
	Certainly $\dim U \leq \dim X$, and we have to show
	that $\dim U \geq \dim X.$
	Write $d := \dim X$, and choose a maximal strictly
	decreasing chain of irreducible closed subsets
	of $X$, say
	$$X = Z_0 \supset Z_1 \supset \dots \supset Z_d.$$
	As noted in the proof of Lemma~\ref{lem:dimension achieved by finite type point},
	the minimal irreducible closed
	subset $Z_d$ is equal to $\{x\}$ for some closed 
	point $x$. 

	If $x \in U,$ then
	$$U = U \cap Z_0  \supset U\cap Z_1 \supset \dots \supset
	U\cap Z_d = \{x\}$$
	is a strictly decreasing chain of irreducible closed 
	subsets of $U$, and so we conclude that $\dim U \geq d$,
	as required.  Thus we may suppose that $x \not\in U.$


	Consider the flat morphism $\Spec \cO_{X,x} \to X$.
	The non-empty (and hence dense) open subset $U$ of $X$
	pulls back to an open subset $V \subset \Spec \cO_{X,x}$.
	Replacing $U$ by a non-empty quasi-compact, and hence
	Noetherian, open subset, we may assume that the inclusion
	$U \to X$ is a quasi-compact morphism.  Since the
	formation of scheme-theoretic images of quasi-compact
	morphisms commutes with flat
	base-change~\cite[\href{http://stacks.math.columbia.edu/tag/081I}{Tag
                 081I}]{stacks-project}), 
	we see that $V$ is dense in $\Spec \cO_{X,x}$,
	and so in particular non-empty,
	and of course $x \not\in V.$  (Here we use $x$ also to denote
	the closed point of $\Spec \cO_{X,x}$, since its image
	is equal to the given point $x \in X$.)
	Now $\Spec \cO_{X,x} \setminus \{x\}$ is 
	Jacobson~\cite[\href{http://stacks.math.columbia.edu/tag/02IM}{Tag
                 02IM}]{stacks-project}, 
	and hence $V$ contains a closed point $z$ 
	of $\Spec \cO_{X,x} \setminus \{x\}$.  The closure
	in $X$ of the image of $z$ is then an irreducible
	closed subset $Z$ of $X$ containing $x$, whose intersection
	with $U$ is non-empty, and
       	for which there is no irreducible closed 
	subset properly contained in $Z$
	and properly containing $\{x\}$
        (because pull-back to $\Spec \cO_{X,x}$ induces
	a bijection between irreducible closed subsets of $X$
	containing $x$ and irreducible closed subsets of $\Spec
	\cO_{X,x}$).
	Since $U \cap Z$ is a non-empty closed subset of $U$,
	it contains a point $u$ that is closed in $X$ (since
	$X$ is Jacobson), and since $U\cap Z$ 
	is a non-empty (and hence dense) open subset of the irreducible set $Z$
	(which contains a point not lying in $U$, namely $x$),
	the inclusion $\{u\} \subset U\cap Z$ is proper.

	As $X$ is catenary, the chain
	$$X = Z_0 \supseteq Z \supset \{x\} = Z_d $$
	can be refined to a chain of length $d+1$, which must then
	be of the form
	$$X = Z_0 \supset W_1 \supset \cdots \supset W_{d-1} = Z \supset \{x\} = Z_d.$$
	Since $U\cap Z$ is non-empty, we then find that
	$$U = U \cap Z_0 \supset U \cap W_1\supset \dots \supset U\cap W_{d-1}
	= U\cap Z \supset \{u\}$$
	is a strictly decreasing chain of irreducible closed subsets
	of $U$ of length $d+1$, showing that $\dim U \geq d$,
	as required.
\end{proof}

We will prove a stack-theoretic analogue
of Lemma~\ref{lem:constancy of dimension} in Lemma~\ref{lem:irreducible
	implies equidimensional} below,
but before doing so, we have to introduce an additional definition,
necessitated by the fact that the notion of a scheme being catenary
is not an \'etale local one
(see
the example
of~\cite[\href{http://stacks.math.columbia.edu/tag/0355}{Tag 0355}]{stacks-project}),
which makes it difficult to define what it means for an algebraic
space or algebraic stack to be catenary
(see the discussion of \cite[p.~3]{MR3351957}).
For certain aspects of dimension theory, the following 
definition seems to provide a good substitute for the missing 
notion of a catenary algebraic stack.

\begin{adf}
	We say that a locally Noetherian algebraic stack $\cX$ 
	is {\em pseudo-catenary} if there exists a smooth
	and surjective morphism $U \to \cX$ whose source is
	a universally catenary scheme.
\end{adf}

\begin{aexample}
	If $\cX$ is locally of finite type over a universally
	catenary locally Noetherian scheme $S$,
	and $U\to \cX$ is a smooth surjective morphism
	whose source is a scheme, then the composite 
	$U \to \cX \to S$ is locally of finite type,
	and so $U$ is universally
	catenary~\cite[\href{http://stacks.math.columbia.edu/tag/02J9}{Tag 02J9}]{stacks-project}.   Thus $\cX$ is pseudo-catenary.
\end{aexample}

The following lemma shows that the property of being pseudo-catenary
passes through finite type morphisms.

\begin{alemma}
	\label{lem:catenary covers}
	If $\cX$ is a pseudo-catenary locally Noetherian algebraic
	stack, and if $\cY \to \cX$ is a locally of finite type morphism,
	then there exists a smooth surjective morphism $V \to \cY$
	whose source is a universally catenary scheme; thus
	$\cY$ is again pseudo-catenary.
\end{alemma}
\begin{proof}
	By assumption we may find a smooth surjective morphism
	$U \to \cX$ whose source is a universally catenary scheme.
	The base-change $U\times_{\cX} \cY$ is then an algebraic
	stack; let $V \to U\times_{\cX} \cY$ be a smooth
	surjective morphism whose source is a scheme.  
	The composite $V \to U\times_{\cX} \cY \to \cY$ is then
	smooth and surjective (being a composite of smooth and
	surjective morphisms), while the morphism $V \to U\times_{\cX}
	\cY \to U$ is locally of finite type (being a composite 
	of morphisms that are locally of finite type).  Since $U$
	is universally catenary, we see that $V$ is universally catenary
(by~\cite[\href{http://stacks.math.columbia.edu/tag/02J9}{Tag 02J9}]{stacks-project}),
	as claimed.
\end{proof}

We now study the behaviour of the function $\dim_x(\cX)$ on $|\cX|$
(for some locally Noetherian stack $\cX$) with respect to the irreducible
components of $|\cX|$, as well as various
related topics.

\begin{alemma}
\label{lem:irreducible implies equidimensional}
If $\cX$ is
a Jacobson, pseudo-catenary, and locally Noetherian  algebraic stack
for which $|\cX|$ is irreducible,
then $\dim_x(\cX)$ is a constant function on~$|\cX|$.
\end{alemma}
\begin{proof}
It suffices to show that $\dim_x(\cX)$ is locally constant on $|\cX|$,
since it will then necessarily be constant (as $|\cX|$ is connected,
being irreducible).  Since $\cX$ is pseudo-catenary,
we may find a smooth surjective morphism $U \to \cX$ with $U$ 
being a univesally catenary scheme.  If $\{U_i\}$ is an 
cover of $U$ by quasi-compact open subschemes, we may replace 
$U$ by $\coprod U_i,$, and 
it suffices to show that
the function $u \mapsto \dim_{f(u)}(\cX)$ is locally constant on $U_i$.
Since we check this for one $U_i$ at a time, we now drop the subscript,
and write simply $U$ rather than~$U_i$.
Since $U$ is quasi-compact, it
is the union of a finite number of irreducible components,
say $T_1 \cup \cdots \cup T_n$.  Note that each $T_i$ is Jacobson,
catenary, and locally Noetherian,
being a closed subscheme of the Jacobson, catenary, and locally Noetherian
scheme~$U$.

By definition, we have $\dim_{f(u)}(\cX) = \dim_{u}(U) - \dim_{u}(U_{f(u)}).$
Lemma~\ref{lem:relative dimension is semi-continuous}~(2)
shows that the second term in the right hand expression is locally
constant 
on $U$, as $f$ is smooth,
and hence we must show that $\dim_u(U)$
is locally constant on $U$.  Since $\dim_u(U)$ is the maximum
of the dimensions $\dim_u T_i$, as $T_i$ ranges over the components
of $U$ containing $u$, it suffices to show
that if a point $u$ lies on two distinct components,
say $T_i$ and $T_j$ (with $i \neq j$),
then $\dim_u T_i = \dim_u T_j$,  
and then to note that $t\mapsto \dim_t T$ is a constant 
function on an irreducible Jacobson,
catenary, and locally Noetherian scheme $T$
(as follows from Lemma~\ref{lem:constancy of dimension}).

Let $V = T_i \setminus \bigl( \bigcup_{i' \neq i} T_{i'}\bigr)$ 
and $W = T_j \setminus \bigl( \bigcup_{i' \neq j} T_{i'}\bigr)$.
Then each of $V$ and $W$ is a non-empty open subset of $U$,
and so each has non-empty open image in $|\cX|$.  As $|\cX|$ is irreducible,
these two non-empty open subsets of $|\cX|$ have a non-empty intersection.
Let $x$ be a point lying in this intersection, and let $v \in V$ and 
$w\in W$ be points mapping to $x$.  
We then find that
$$\dim T_i = \dim V = \dim_v (U) = \dim_x (\cX) + \dim_v (U_x)$$
and similarly that
$$\dim T_j = \dim W = \dim_w (U) = \dim_x (\cX) + \dim_w (U_x).$$
Since $u \mapsto \dim_u (U_{f(u)})$ is locally constant on $U$,
and since $T_i \cup T_j$ is connected (being the union of two irreducible,
hence connected, sets that have non-empty intersection), 
we see that $\dim_v (U_x) = \dim_w(U_x)$, 
and hence, comparing the preceding two equations,
that $\dim T_i = \dim T_j$, as required.
\end{proof}

\begin{alemma}
	\label{lem:closed immersions}
	If $\cZ \hookrightarrow \cX$ is a closed immersion
	of locally Noetherian schemes,
	and if $z \in |\cZ|$ has image $x \in |\cX|$,
	then $\dim_z (\cZ) \leq \dim_x(\cX)$.
\end{alemma}
\begin{proof}
	Choose a smooth surjective morphism
	$U\to \cX$ whose source is a scheme;
	the base-changed morphism $V := U\times_{\cX} \cZ \to \cZ$
	is then also smooth and surjective, and the projection
	$V \to U$ is a closed immersion.
	If $v \in |V|$ maps to $z \in |\cZ|$, and
	if we let $u$ denote the image of $v$ in $|U|$,
	then clearly
	$\dim_v(V) \leq \dim_u(U)$,
	while
	$\dim_v (V_z) = \dim_u(U_x)$,
	by Lemma~\ref{lem:base-change invariance of relative dimension}.
	Thus $$\dim_z(\cZ)  = \dim_v(V) - \dim_v(V_z)
	\leq \dim_u(U) - \dim_u(U_x) = \dim_x(\cX),$$
	as claimed.  
\end{proof}

\begin{alemma}
\label{lem:dimension via components}
If $\cX$ is a locally Noetherian algebraic stack, and if $x \in |\cX|$,
then $\dim_x(\cX) = \sup_{\cT} \{ \dim_x(\cT) \} $,
where $\cT$ runs over all the irreducible components 
of $|\cX|$ passing through $x$ {\em (}endowed with their
induced reduced structure{\em )}. 
\end{alemma}
\begin{proof}
	Lemma~\ref{lem:closed immersions}
	shows that
	$\dim_x (\cT) \leq \dim_x(\cX)$ for each
	irreducible component $\cT$ passing through
	the point $x$.   Thus to prove the lemma,
	it suffices to show that
	\anumequation
	\label{eqn:desired inequality}
	\dim_x(\cX) \leq 
	\sup_{\cT} \{\dim_x(\cT)\}.
\end{equation}

Let $U\to\cX$ be a smooth cover by a scheme. If~$T$ is an irreducible
component of~$U$ then we let~$\cT$ denote the closure of its image
in~$\cX$, which is an irreducible component of~$\cX$. Let~$u\in U$ be
a point mapping to~$x$. Then we have
$\dim_x(\cX)=\dim_uU-\dim_uU_x=\sup_T\dim_uT-\dim_uU_x$, where 
the supremum is over the irreducible components of~$U$ passing
through~$u$. Choose a component $T$ for which the supremum
is achieved, and note that
$\dim_x(\cT)=\dim_uT-\dim_u T_x$.
The desired inequality~(\ref{eqn:desired inequality})
now follows from the evident inequality $\dim_u T_x \leq \dim_u U_x.$
(Note that if $\Spec k \to \cX$ is a representative of $x$,
then $T\times_{\cX} \Spec k$ is a closed subspace of $U\times_{\cX}
\Spec k$.) 
\end{proof}

\begin{alemma}
If $\cX$ is a locally Noetherian algebraic stack, and if $x \in |\cX|$, then
for any open substack $\cV$ of $\cX$ containing $x$,
there is a finite type point $x_0 \in |\cV|$ such that
$\dim_{x_0}(\cX) = \dim_x(\cV)$.
\end{alemma}
\begin{proof}
Choose a smooth surjective
morphism  $f:U \to \cX$ whose source is a scheme, and consider the
function $u \mapsto \dim_{f(u)}(\cX);$
since the morphism $|U| \to |\cX|$ induced by $f$ is open (as $f$ 
is smooth) as well as  surjective (by assumption),
and takes finite type points to finite type points (by the very definition
of the finite type points of $|\cX|$), 
it suffices to show that for any $u \in U$, and any open neighbourhood of $u$,
there is a finite type point $u_0$ in this neighbourhood such that
$\dim_{f(u_0)}(\cX) = \dim_{f(u)}(\cX).$ 
Since, with this reformulation
of the problem,  the surjectivity of $f$ is no longer required, 
we may replace $U$ by the open neighbourhood of the point $u$ in question,
and thus reduce to the problem of showing that for each $u \in U$,
there is a finite type point $u_0 \in U$ such that
$\dim_{f(u_0)}(\cX) = \dim_{f(u)}(\cX).$ 
By definition
$\dim_{f(u)}(\cX) = \dim_u(U) - \dim_u(U_{f(u)}),$
while 
$\dim_{f(u_0)}(\cX) = \dim_{u_0}(U) - \dim_{u_0}(U_{f(u_0)}).$
Since $f$ is smooth, the expression $\dim_{u_0}(U_{f(u_0)})$ is locally
constant as $u_0$ varies over $U$ (by Lemma~\ref{lem:relative dimension 
is semi-continuous}~(2)), and so shrinking $U$ further around
$u$ if necessary, we may assume it is constant.  Thus the problem
becomes to show that we may find a finite type point $u_0 \in U$
for which $\dim_{u_0}(U) = \dim_u(U)$. 
Since by definition $\dim_u U$ is the minimum of the dimensions
$\dim V$, as $V$ ranges over the open neighbourhoods $V$ of $u$
in $U$, we may shrink $U$ down further around $u$ so that
$\dim_u U = \dim U$.
The existence of desired point $u_0$ then follows from
Lemma~\ref{lem:dimension achieved by finite type point}. 
\end{proof}


\begin{alemma}\label{lem: monomorphing a component in of the right dimension}
	Let $\cT \hookrightarrow \cX$ be a locally
of finite type monomorphism of algebraic stacks,
with $\cX$ {\em (}and thus also $\cT${\em )} 
being Jacobson, pseudo-catenary, and locally Noetherian.
Suppose further that $\cT$ is irreducible
of some {\em (}finite{\em )} dimension~$d$, and that $\cX$ is reduced and of dimension less
than or equal to $d$.  
Then there is a non-empty open substack $\cV$ of $\cT$ such that the induced
monomorphism $\cV \hookrightarrow \cX$ is an open immersion which identifies 
$\cV$ with an open subset of an irreducible component of~$\cX$.
\end{alemma}
\begin{proof}
Choose a smooth surjective morphism $f:U \to \cX$ with source a scheme,
necessarily reduced since $\cX$ is,
and write $U' := \cT\times_{\cX} U$.  The base-changed morphism
$U' \to U$ is a monomorphism of algebraic spaces, locally of finite
type, and thus representable~\cite[\href{http://stacks.math.columbia.edu/tag/0418}{Tag 0418},
\href{http://stacks.math.columbia.edu/tag/0463}{Tag 0463}]{stacks-project}; since $U$ is a scheme, so is~$U'$.
The projection $f': U' \to \cT$ is again a smooth surjection.
Let $u' \in U'$, with image $u \in U$.
Lemma~\ref{lem:base-change invariance of relative dimension}
shows that $\dim_{u'}(U'_{f(u')}) = \dim_u(U_{f(u)}),$
while $\dim_{f'(u')}(\cT) =d
\geq \dim_{f(u)}(\cX)$ by Lemma~\ref{lem:irreducible implies equidimensional}
and our assumptions on~$\cT$ and~$\cX$.
Thus we see that
\anummultline
\label{eqn:dim inequality}
\dim_{u'} (U') = \dim_{u'} (U'_{f(u')}) + \dim_{f'(u')}(\cT) 
\\
\geq \dim_u (U_{f(u)}) + \dim_{f(u)}(\cX) = \dim_u (U).
\end{multline}
Since $U' \to U$ is a monomorphism, locally of finite type,
it is in particular unramified,
and so by the \'etale local structure of unramified morphisms~\cite[\href{http://stacks.math.columbia.edu/tag/04HJ}{Tag 04HJ}]{stacks-project}, we may find
a commutative diagram
$$\xymatrix{V' \ar[r]\ar[d] & V \ar[d] \\
U' \ar[r] & U}$$
in which the scheme $V'$ is non-empty,
the vertical arrows are \'etale, 
and the upper horizontal arrow is a closed immersion. 
Replacing $V$ by a quasi-compact open subset
whose image has non-empty intersection with the image of $U'$,
and replacing $V'$ by the preimage of $V$, we may further
assume that $V$ (and thus $V'$) is quasi-compact. 
Since $V$ is also locally Noetherian,
it is thus Noetherian, and so is the union of finitely many irreducible 
components. 

Since \'etale morphisms preserve pointwise
dimension~\cite[\href{http://stacks.math.columbia.edu/tag/04N4}{Tag
  04N4}]{stacks-project},  we deduce from~(\ref{eqn:dim inequality})
that for any point $v' \in V'$,
with image $v \in V$, we have
$\dim_{v'}( V') \geq \dim_v(V)$. 
In particular, the image of $V'$ can't be contained in the intersection
of two distinct irreducible components of $V$, and so we may find 
at least one irreducible open subset of $V$ which has non-empty intersection
with $V'$; replacing $V$ by this subset, we may assume that $V$ is integral
(being both reduced and irreducible).  From the preceding inequality 
on dimensions, we conclude that the closed immersion $V' \hookrightarrow V$
is in fact an isomorphism.
If we let $W$ denote the image of $V'$
in~$U'$, then $W$ is a non-empty
open subset of $U'$ (as \'etale morphisms are open),
and the induced monomorphism $W \to U$ is \'etale
(since it is so \'etale locally on the source, i.e.\ after pulling back
to $V'$), and hence is an open immersion (being an \'etale monomorphism).
Thus, if we let $\cV$ denote the image of $W$ in $\cT$,
then $\cV$ is a dense (equivalently, non-empty) open substack of $\cT$,
whose image is dense in an irreducible component of~$\cX$.
Finally,
we note that the morphism is $\cV \to \cX$ is smooth (since its composite
with the smooth morphism $W\to \cV$ is smooth),
and also a monomorphism, and thus is an open immersion.
\end{proof}

\begin{alemma}
	\label{lem:dims of images}
	Let $f: \cT \to \cX$ be a locally of finite type
	morphism of Jacobson, pseudo-catenary, and locally Noetherian 
	algebraic stacks, 
	whose source is irreducible and whose target is quasi-separated,
	and let $\cZ \hookrightarrow \cX$ denote the scheme-theoretic
	image of $\cT$.
	Then for every finite type point $t \in |T|$, 
	we have that $\dim_t( \cT_{f(t)}) \geq \dim \cT  - \dim \cZ,$
	and there is a non-empty {\em (}equivalently, dense{\em )}
	open subset of $|\cT|$ over which equality holds.
\end{alemma}
\begin{proof}
        %
        %
	Replacing $\cX$ by $\cZ$, 
	we may and do assume that $f$ is scheme theoretically dominant,
	and also that $\cX$ is irreducible.
	By the upper semi-continuity of fibre dimensions
	(Lemma~\ref{lem:relative dimension is semi-continuous}~(1)),
	it suffices to prove that the equality $\dim_t( \cT_{f(t)}) =\dim \cT  - \dim \cZ$ holds for $t$ lying in
	some non-empty open substack of $\cT$.
	For this reason, in the argument we are always free 
	to replace $\cT$ by a non-empty open substack.

	Let $T' \to \cT$ be a smooth surjective morphism whose source
	is a scheme, and let $T$ be a non-empty quasi-compact open subset
	of $T'$.  Since $\cY$ is quasi-separated, we find 
	that $T \to  \cY$ is quasi-compact
          (by~\cite[\href{http://stacks.math.columbia.edu/tag/050Y}{Tag
            050Y}]{stacks-project}, applied to the morphisms
    $T \to \cY \to \Spec \Z$).
	Thus, if we replace $\cT$ by the image of $T$ in $\cT$,
	then we may assume (appealing 
          to~\cite[\href{http://stacks.math.columbia.edu/tag/050X}{Tag
            050X}]{stacks-project})
	that the morphism $f:\cT \to \cX$ is quasi-compact.

	If we choose a smooth surjection $U \to \cX$ with $U$ a scheme,
	then Lemma~\ref{lem:map of components} ensures that
	we may find an irreducible open subset $V$ of $U$ such
	that $V \to \cX$ is smooth and scheme-theoretically dominant.
	Since scheme-theoretic dominance for quasi-compact morphisms
	is preserved by flat base-change,
	the base-change $\cT \times_{\cX} V \to V$ of the scheme-theoretically
	dominant morphism $f$ is again
	scheme-theoretically dominant.   We let $Z$ denote a scheme
	admitting a smooth surjection onto this fibre product;
	then $Z \to \cT \times_{\cX} V \to V$ 
	is again scheme-theoretically dominant.
	Thus we may find an irreducible
	component $C$ of $Z$ which scheme-theoretically
	dominates~$V$.  
	Since the composite  $Z \to \cT\times_{\cX} V \to \cT$ is smooth,
	and since $\cT$ is irreducible,
	Lemma~\ref{lem:map of components} shows that any irreducible
	component of the source has dense image in $|\cT|$. 
	We now replace
	$C$ by a non-empty open subset $W$ which is disjoint from every other
	irreducible component of~$Z$, and 
	then replace $\cT$ and $\cX$ by the images of $W$ 
	and $V$
	(and apply Lemma~\ref{lem:irreducible implies equidimensional}
	to see that this
	doesn't change the dimension of either $\cT$ or $\cX$).
        If we let $\cW$ denote the image of the morphism
	$W \to \cT\times_{\cX} V$, 
	then $\cW$ is open in $\cT\times_{\cX} V$ (since the
	morphism $W \to \cT\times_{\cX} V$ is smooth),
	and is irreducible (being the image of an irreducible
	scheme).  Thus we end up with a commutative diagram
	$$\xymatrix{W \ar[dr] \ar[r]  & \cW \ar[r] \ar[d]
		& V \ar[d] \\ & \cT \ar[r] & \cX}$$
	in which $W$ and $V$ are schemes,
	the vertical arrows are smooth and surjective,
	the diagonal arrows and the left-hand 
        upper horizontal arrow	are smooth, 
	and the induced morphism $\cW \to \cT\times_{\cX} V$ is an
	open immersion.
	Using this diagram, together with the definitions
	of the various dimensions involved in
        the statement of the lemma, we will reduce our verification
	of the lemma to the case of schemes, where it is known.

        Fix $w \in |W|$ with image $w' \in |\cW|$,
	image $t \in |\cT|$, image $v$ in $|V|$,
	and image $x$ in~$|\cX|$.
	Essentially by definition (using the
	fact that $\cW$ is open in $\cT\times_{\cX} V$, and that
	the fibre of a base-change is the base-change of the fibre),
	we obtain the equalities
	$$\dim_v V_x = \dim_{w'} \cW_t $$
	and
        $$ \dim_t \cT_x = \dim_{w'} \cW_v.$$
        Again by definition (the diagonal arrow and right-hand vertical
	arrow in our diagram realise $W$ and $V$ as smooth covers by 
	schemes of the stacks $\cT$ and $\cX$), we find that 
	$$\dim_t \cT = \dim_w W - \dim_w W_t $$
	and
	$$ \dim_x \cX = \dim_v V - \dim_v V_x.$$
	Combining the equalities, we find that
	\begin{multline*}
	\dim_t \cT_x - \dim_t \cT + \dim_x \cX
	\\
	= \dim_{w'} \cW_v - \dim_w W + \dim_w W_t + \dim_v V - \dim_{w'} \cW_t.
	\end{multline*}
	Since $W \to \cW$ is a smooth surjection, the same is true 
	if we base-change over the morphism $\Spec \kappa(v) \to V$
	(thinking of $W \to \cW$ as a morphism over $V$),
	and from this smooth morphism we obtain the first of the following
	two equalities
	$$\dim_w W_v - \dim_{w'} \cW_v = \dim_w (W_v)_{w'} = \dim_w W_{w'};$$
	the second equality follows via a direct comparison of the
	two fibres involved.
	Similarly, if we think of $W \to \cW$ as a morphism of schemes
	over $\cT$, and base-change over some representative of the point
	$t \in |\cT|$, we obtain the equalities
	$$\dim_w W_t - \dim_{w'} \cW_t = \dim_w (W_t)_{w'} = \dim_w W_{w'}.$$
	Putting everything together, we find that
	$$
	\dim_t \cT_x - \dim_t \cT + \dim_x \cX
	=  \dim_w W_v - \dim_w W + \dim_v V.
	$$
	Our goal is to show that the left-hand side of this equality
        vanishes for a non-empty open subset
	of $t$.  As $w$ varies over a non-empty open subset of $W$,
	its image $t \in |\cT|$ varies over a non-empty open
	subset of $|\cT|$ (as $W \to \cT$ is smooth).  

We are therefore reduced to showing that if $W\to V$ is a
scheme-theoretically dominant morphism of irreducible locally
Noetherian schemes that is locally of finite type,
then there is a non-empty open subset of
points $w\in W$ such that $\dim_w W_v =\dim_w W - \dim_v V$
(where $v$ denotes the image of $w$ in $V$).
This is a standard fact,
whose proof we recall for the convenience of the reader.

We may replace $W$ and $V$ by their underlying reduced subschemes
without altering the validity (or not) of this equation,
and thus we may assume that they are in fact integral schemes. 
Since $\dim_w W_v$ is locally constant on $W,$ replacing $W$ 
by a non-empty open subset if necessary, we may assume that $\dim_w W_v$
is constant, say equal to $d$.  Choosing this open subset to be affine,
we may also assume that the morphism $W\to V$ is in fact of finite type.
Replacing $V$ by a non-empty open subset if necessary
(and then pulling back $W$ over this open subset; the resulting pull-back
is non-empty, since the flat base-change of a quasi-compact 
and scheme-theoretically
dominant morphism remains scheme-theoretically dominant),
we may furthermore assume that $W$ is flat over $V$.
The morphism $W\to V$ is thus of relative dimension $d$
in the sense
          of~\cite[\href{http://stacks.math.columbia.edu/tag/02NJ}{Tag
            02NJ}]{stacks-project},
   and it follows 
          from~\cite[\href{http://stacks.math.columbia.edu/tag/0AFE}{Tag
            0AFE}]{stacks-project} that
   $\dim_w(W) = \dim_v(V) + d,$ as required.
\end{proof}

\begin{aremark}
	We note that in the context of the preceding lemma,
	it need not be that $\dim \cT \geq \dim \cZ$; this does
	not contradict the inequality in the statement of the lemma, because
	the fibres of the morphism $f$ are again algebraic stacks, and
	so may have negative dimension.  This is illustrated by taking
	$k$ to be a field, and applying the lemma to the morphism
	$[\Spec k/\Gm] \to \Spec k$.

	If the morphism $f$ in the statement of the lemma is assumed
	to be quasi-DM (in the sense
        of~\cite[\href{http://stacks.math.columbia.edu/tag/04YW}{Tag
        04YW}]{stacks-project}; e.g.\ morphisms that are
        representable by algebraic spaces are quasi-DM),
	then the fibres of the morphism over points of the target
	are quasi-DM algebraic stacks, and hence are of non-negative
	dimension.  In this case, the lemma implies
	that indeed $\dim \cT \geq \dim \cZ$.  In fact, we obtain
	the following more general result.
\end{aremark}

\begin{acor}
	\label{cor:dims of images}
	Let $f: \cT \to \cX$ be a locally of finite type
	morphism of Jacobson, pseudo-catenary, and locally Noetherian 
	algebraic stacks
        which is quasi-DM,
	whose source is irreducible and whose target is quasi-separated,
	and let $\cZ \hookrightarrow \cX$ denote the scheme-theoretic
	image of $\cT$.
	Then $\dim \cZ \leq \dim \cT$,
	and furthermore, exactly one of the following two conditions holds:
	\begin{enumerate}
		\item for every finite type point $t \in |T|,$
			we have
			$\dim_t(\cT_{f(t)}) > 0,$ in which
			case $\dim \cZ < \dim \cT$; or
		\item   $\cT$ and $\cZ$
			are of the same dimension.
	\end{enumerate}
\end{acor}
\begin{proof}
	As was observed in the preceding remark,
	the dimension of a quasi-DM stack is always non-negative,
	from which we conclude that $\dim_t \cT_{f(t)} \geq 0$
	for all $t \in |\cT|$, with the equality
	$$\dim_t \cT_{f(t)} = \dim_t \cT - \dim_{f(t)} \cZ$$ holding
	for a dense open subset of points $t\in |\cT|$.
\end{proof}

We close this note by establishing a formula allowing us to
compute $\dim_x(\cX)$ in terms of properties of the versal ring
to $\cX$ at $x$.
In order to state a clean result,
we will make certain hypotheses on the base-scheme $S$
(which has remained implicit up to this point).
As with the discussion at the
end of Section~\ref{sec:multiplicities}, these hypotheses may not be 
needed for the result to hold, but they allow for a simple 
argument.

Before stating our hypotheses, 
we recall some topological results.  These results are essentially
contained in~\cite[\S 0.14.3]{MR0173675}. 
However, as is pointed out in~\cite{MR3631826},
the key proposition of that discussion, namely~\cite[\S 0 Prop.~14.3.3]{MR0173675},
is in error.  As is made implicit in the examples of \cite{MR3631826},
and was pointed out explicitly to us by Brian Conrad, 
the error occurs because a topological space can be equicodimensional
without its irreducible components themselves being equicodimensional.
We now state (and recall the proof of) a corrected version of that
proposition (and of Cor.~14.3.5, which is deduced from
Prop.~14.3.3).

We first recall the definition of equicodimensionality.

\begin{adf}
	A finite-dimensional topological space $X$
	is called {\em equicodimensional} if $\codim(Y,X)$
	is constant as $Y$ ranges over all the minimal 
	irreducible closed subsets of $X$.
	(By considering a maximal chain of irreducible closed
	subsets of $X$, we then see that this constant is equal
	to $\dim X$.)
\end{adf}

\begin{alemma}
	\label{lem:dimension lemma}
	If $X$ is an irreducible, equicodimensional,
	finite-dimensional topological
	space, then the following are equivalent:

	(1)  $X$ is catenary.

	(2) For any irreducible closed subsets $Y\subseteq Z$
	of $X$,
	we have $$\dim Y + \codim(Y,Z) = \dim Z.$$

	(3) All maximal chains of irreducible closed subsets
	of $X$ have the same length.

	Furthermore, if these equivalent conditions hold,
	then any irreducible closed subset of $X$ is also
	equicodimensional.
\end{alemma}
\begin{proof}
	Let $Y$ be an irreducible closed subset of $X$,
	and consider maximal chains of irreducible closed sets
	\numequation
	\label{eqn:Y chain}
	Y_0 \subset Y_1 \subset \cdots \subset Y_a = Y
\end{equation}
        and 
	\numequation
	\label{eqn:X chain}
	Y = X_0 \subset X_1 \subset \cdots \subset X_b = X
\end{equation}
Since~(\ref{eqn:Y chain}) is maximal, we see that $Y_0$ is a
minimal irreducible closed subset of $Y$ (or, equivalently, of $X$).
Concatenating these two chains yields a maximal chain of
irreducible closed subsets in $X$. 

Suppose now that (1) holds, i.e.\ that $X$ is catenary.  Then
all maximal chains joining $Y_0$ to $X$ have
the same length, which is then $\codim(Y_0,X)$, which also
equals $\dim X$ (since $X$ is equicodimensional, by assumption).   Thus we find that
$a + b = \dim X$, and in particular is independent of the choice
of either chain.  Varying~(\ref{eqn:Y chain}), while leaving
(\ref{eqn:X chain}) fixed, we find that the value of $a$ is independent
of the choice of the maximal chain~(\ref{eqn:Y chain}).  Thus we see
that $a = \dim Y$, and also that $\codim(Y_0,Y) = \dim Y$ for any
minimal irreducible closed subset of $Y$ (so that $Y$ is again
equicodimensional, as claimed).

Now fixing the maximal chain~(\ref{eqn:Y chain}), and varying the
maximal chain~(\ref{eqn:X chain}), we find that the value of $b$
is independent of the choice of chain~(\ref{eqn:X chain}), 
and in particular that $b = \codim(Y,X)$.  Thus we may rewrite the
equation $a + b = \dim X$ as $\dim Y + \codim(Y,X) = \dim X$, 
showing that~(2) holds in the case when $Z = X$.  
If we consider the general case of~(2),
then since $Z$ is irreducible (by assumption), finite-dimensional 
and catenary (being a closed subset of a finite-dimensional and catenary space),
and equicodimensional (as we proved above), we may replace $X$ by $Z$,
and hence deduce the general case of (2) from the special case
already proved.

Suppose next that~(2) holds,  and consider a maximal chain
of irreducible closed subsets of $X$, say
$$X_0 \subset X_1 \subset \cdots \subset X_d = X.$$
Noting that $\dim X_0 = 0$ (as $X_0$ is minimal), and
also that $\codim(X_i,X_{i+1}) = 1$ (since by assumption
there is no irreducible closed subset lying strictly between $X_i$ 
and $X_{i+1}$), we find, by repeated application of~(2), 
that $\dim X_i = i$.  In particular, $d = \dim X$ is independent
of the chain chosen, so that~(3) holds.

Finally, suppose that~(3) holds;  we wish to show that~(1) also holds.
If we consider chains
of the form (\ref{eqn:Y chain}) and~(\ref{eqn:X chain}), and
their concatenation, then~(3) implies that $a + b = \dim X$ is indepedent
of the choice of either chain, and thus, by varying these chains
independently, that each of $a$ and $b$ is independent of the
choice of chain.

Now let $Y \subseteq Z$ be an inclusion of irreducible closed
subsets of $X$.  We wish to show that all maximal chains of irreducible
closed subsets joining $Y$ and $Z$ are of the same length.
By applying what we have just proved to $Z$,
we find that $Z$ also satisfies~(3).  Thus we may replace $X$ by $Z$,
and hence assume that $Z = X$.  But we have already shown that
all maximal chains of the form~(\ref{eqn:X chain}) are of the same
length.  Thus $X$ is indeed catenary.
\end{proof}

Although we don't need it, 
we also note the following result,
which among other things
provides a purely topological variant of Lemma~\ref{lem:constancy
of dimension}.
(Note, though, that~\cite[(10.7.3)]{MR0217086} gives an example of a Jacobson,
universally catenary, integral, Noetherian scheme $S$ 
which is not equicodimensional;
this gives an example
of a situation to which Lemma~\ref{lem:constancy of dimension} applies,
although Lemma~\ref{lem:equicodimensionality of opens} does not.)

\begin{alemma}
	\label{lem:equicodimensionality of opens}
	Let $X$ be an irreducible, equicodimensional, finite-dimensional,
	Jacobson topological space.  If $U$ is a non-empty open
	subset of $X$, then $U$ is also irreducible, equicodimensional,
	finite-dimensional, and Jacobson.  Furthermore, we have
	that $\dim U = \dim X$.
\end{alemma}
\begin{proof}
	It is standard that $U$ is again irreducible and Jacobson.
	The function $T\mapsto \overline{T}$ (closure in $X$)
	induces an order-preserving bijection between irreducible closed 
	subsets of $U$ and irreducible closed subsets of $X$ 
	that have non-empty intersection with $U$; thus $U$ is
	certainly also finite-dimensional.

	Since $X$ and $U$ are Jacobson, the minimal irreducible
	closed subsets of either $X$ or $U$ are just
	the closed points, and the closed points of $U$ are precisely
        the closed points of $X$ that lie in $U$.	
	Thus, under the bijection $T\mapsto \overline{T}$ described
	above, the collection of
	maximal chains of irreducible closed subsets of $U$
	containing some given closed point $u \in U$
	maps bijectively to the collection of
	maximal chains of irreducible closed subsets of $X$
	containing the same closed point $u$.  In particular,
        we find that
	$$\codim(u,U) = \codim(u,X) = \dim X$$
	(the last equality hoding 
       	since $X$ is equidimensional, by assumption.)
	We thus see that $\codim(u,U)$ is independent of the particular
	closed point $u \in U$, so~$U$ is equicodimensional.  Furthermore, it is then
	necessarily equal to $\dim U$,
	and so we also find that $\dim U   = \dim X$, as claimed.
\end{proof}

We next note the following scheme-theoretic result.

\begin{alemma}
	\label{lem:comparing dimensions}
	If $S$ is a Jacobson, catenary, locally Noetherian scheme,
	all of whose 
	irreducible components are of finite 
	dimension and equicodimensional,
	and if $s \in S$ is a finite type point {\em (}or equivalently, 
	a closed point, by Jacobsonness{\em )},
	then $\dim_s S = \dim \mathcal O_{S,s}.$
\end{alemma}
\begin{proof}
	We have the equality $\dim \mathcal O_{S,s}  = \codim(s,S)$~\cite[\href{http://stacks.math.columbia.edu/tag/02IZ}{Tag 02IZ}]{stacks-project}.
	If we let $T_1,\ldots, T_n$ denote the irreducible
	components of $S$ passing through $s$, then
	$\codim(s,S) = \max_{i = 1,\ldots,n} \codim(s,T_i),$
	and similarly, $\dim_s S = \max_{i=1,\ldots,n} \dim_s T_i.$
	Thus it suffices to show that
	$\codim(s,T_i) = \dim_s T_i$ for each $T_i$.
	This follows from Lemma~\ref{lem:constancy of dimension},
	which shows that $\dim_s T_i = \dim T_i$,
	together with the assumption that $T_i$ is equicodimensional.
%
\end{proof}
 
We now state the hypothesis that we will make on our base scheme $S$.


\begin{ahyp}
	\label{hyp:good hypotheses}
       	We assume that $S$ is a Jacobson, universally
	catenary, locally Noetherian scheme, all of whose 
	local rings 
	are $G$-rings,
	and with the further property that each 
	irreducible component 
	of $S$ is of finite dimension
	and equicodimensional.
\end{ahyp}

\begin{aremark}
	\label{rem:good hypotheses}
	Since $S$ is catenary by assumption,
	we see that the equivalent conditions of Lemma~\ref{lem:dimension
		lemma} hold for $S$.
	The conditions of Lemma~\ref{lem:equicodimensionality of opens}
	also hold.  Combining these lemmas, we find 
	in particular that each irreducible locally closed subset of $T$
	is equicodimensional.  
	Since $S$ is Jacobson, so  is its locally closed
        subset $T$.
	The finite type points in $T$ are
	then the same as the closed points~\cite[\href{http://stacks.math.columbia.edu/tag/01TB}{Tag 01TB}]{stacks-project},
	and these are also the minimal irreducible closed subsets of $T$.
	Thus to say that $T$ is equicodimensional is to
	say that $\codim(t,T)$ is constant (equal to $\dim T$)
	as $t$ ranges over all closed points of $T$.
\end{aremark}

\begin{alemma}
	\label{lem:transferring hypothesis through ft maps}
	If $X \to S$ is a locally finite type morphism
	of schemes, and if $S$ satisfies Hypothesis~{\em \ref{hyp:good
		hypotheses}},
	then so does $X$.
\end{alemma}
\begin{proof}
	The properties of being Jacobson, of being universally catenary,
	and of the local rings 
	being $G$-rings,
        all pass through a finite type morphism. (In the case of being
        universally catenary, this is immediate from the definition;
        for Jacobson
        see~\cite[\href{http://stacks.math.columbia.edu/tag/02J5}{Tag
          02J5}]{stacks-project}; and for local rings 
  being $G$-rings, see \cite[\href{http://stacks.math.columbia.edu/tag/07PV}{Tag 07PV}]{stacks-project}). 
      Suppose then that $T$ is an irreducible component
      of~$X$; we must show that $T$ is finite-dimensional
      and equicodimensional.

We regard $T$ as an integral scheme, by endowing it with its
induced reduced structure. The composite $T\to X \to S$
	is again locally of finite type, and so replacing $X$ by~$T$,
and $S$ by the closure of the image of $T$, also endowed with its reduced
induced structure (note that by Remark~\ref{rem:good hypotheses},
this closure is again equicodimensional, and since closed immersions
are finite type, the discussion of the preceding paragraph shows
that it also satisfies the other conditions of Hypothesis~\ref{hyp:good
hypotheses}),
	we may assume that each of $X$ and $S$ are integral.
We now have to show that $X$ is finite-dimensional,
and that $\codim(x,X)$ is independent of the closed point $x \in
        X$.


If $x \in X$, we may find an affine neighbourhood $U$ of $x$,
as well as an affine open subset $V\subset S$ containing 
the image of $U$ in $S$.   By assumption $V$ is finite-dimensional,
and $U$ is of finite type over $V$; thus $U$ is also finite-dimensional.
Since $U$ is furthermore irreducible, Jacobson, catenary, and 
locally Noetherian (being open in the irreducible,
Jacobson, catenary, locally Noetherian scheme $X$), we see
from Lemma~\ref{lem:constancy of dimension} that
the function $u \mapsto \dim_u U$ is constant on $U$.  Since $U$ is
open in $X$, we have an equality $\dim_u U = \dim_u X$ for each point
$u \in U$, and hence the function $u \mapsto \dim_u X$ is constant on $U$.
Thus each point of $X$ has a neighbourhood over which $\dim_x X$ is 
constant (and finite valued).  Thus $\dim_x X$ is a locally constant
finite valued function on $X$.  Since $X$ is irreducible (and so in particular
connected) we find that $\dim_x X$ is constant (and finite
valued), and consquently that $X$ is finite-dimensional.

We turn to proving that $X$ equicodimensional.  To this end,
let $x \in X$ be a closed point.
Since $S$ is universally catenary,
        the dimension formula~\cite[\href{http://stacks.math.columbia.edu/tag/02JU}{Tag 02JU}]{stacks-project},
	together with the formula of~\cite[\href{http://stacks.math.columbia.edu/tag/02IZ}{Tag 02IZ}]{stacks-project},
	shows that
	$$\codim(x,X) = \codim(s,S) + \trdeg_{R(S)}R(X) + 
	\trdeg_{\kappa(s)}\kappa(x);$$
	here $s$ denotes the image of $x$
	in $S$, which, being a finite type point of the Jacobson
        scheme $S$, is closed in $S$, and~$R(X)$ (resp.\ $R(S)$) is
        the function field of~$X$ (resp.\ ~$S$).
	Since $\Spec \kappa(x) \to \Spec \kappa(s)$
	is of finite type,
	we have that $\kappa(x)$ is a finite extension of $\kappa(s)$,
	so that the final term on the right-hand side of the formula vanishes.
	Thus $\codim(x,X) -\codim(s,S)$ is constant (i.e.\ independent
	of the closed point $x \in X$).   Since $S$ is equicodimensional,
	the term $\codim(s,S)$ is also independent
	of the closed point $s \in S$; thus $\codim(x,X)$ is indeed
	independent of the closed point $x \in X$.
\end{proof}


We are now able to state and prove the following result,
which relates the dimension of an algebraic stack at a point
to the dimension of a corresponding versal ring.

\begin{alem}
  \label{lem: dimension formula}Suppose that $\cX$ is an algebraic
  stack, locally of finite presentation over a 
  scheme $S$ which satisfies Hypothesis~{\em \ref{hyp:good hypotheses}}.  
  Suppose further
  that $x:\Spec k\to\cX$ is a morphism whose source is the spectrum
  of a field of finite type over $\cO_S$, and that
 $[U/R]\iso\widehat{\mathcal{X}}_x$ is a presentation of $\widehat{\cX}_x$
 by a smooth groupoid in functors,
 with~$U$ and $R$ both Noetherianly pro-representable\footnote{We say
	 that a functor on the category of $S$-schemes is Noetherianally
	 pro-representable if is isomorphic to a functor
	 of the form $\Spf A$, where $A$ is a complete Noetherian local
	 $\cO_S$-algebra equipped with its $\mathfrak m$-adic topology; 
	 here $\mathfrak m$ is the maximal ideal of $A$, of course.}, by
  $\Spf A_x$ and $\Spf B_x$ respectively. Then we have the following formula:
  $$2\dim A_x-\dim B_x=\dim_x(\cX).$$
\end{alem}
\begin{proof}
	By Lemma~\ref{lem: Artin approximation by smooth morphism},
	we may find a smooth morphism $V \to \cX$, whose source is a scheme,
	containing a point $v \in V$ of residue field $k$, such that induced
	morphism $v = \Spec k \to V \to \cX$ coincides with $x$,
	and such that $\widehat{\cO}_{V,x}$ may be identified with $A_x$.
	If we write $W :=  V\times_{\cX} V,$ and we write $w := (v,v) \in W,$
	then we may furthermore identify $\widehat{\cO}_{W,w}$ with $B_x$.
	Now Remark~\ref{rem:computing dims} shows that
	$$\dim_x{\cX} = \dim_v V - \dim_{w}(W_v) = \dim_v V - 
	(\dim_w W - \dim_v V) = 2\dim_v V - \dim_w W.$$
	Since $v$ is a finite type point of $V$, 
	we have that
	$\dim_v V = \dim \cO_{V,v} = \dim \widehat{\cO_{V,v}} = \dim A_x$
	(where we apply
	Lemmas~\ref{lem:transferring hypothesis through ft maps}
	and~\ref{lem:comparing dimensions} to obtain the first equality),
	and similarly $\dim_w W = \dim B_x$.  Thus the formula of the lemma
	is proved.
\end{proof}

\printbibliography

\end{document}